\documentclass[10pt]{amsart}

\usepackage[utf8]{inputenc}
\usepackage[left=4 cm, right= 4 cm]{geometry}
\usepackage{lmodern}
\usepackage[english,activeacute]{babel}
\usepackage{empheq}
\usepackage{xcolor}
\usepackage{amsmath}
\usepackage{amsthm}
\usepackage{amssymb}
\usepackage{calligra,mathrsfs}
\usepackage{enumitem}
\usepackage{extpfeil}
\usepackage{graphics}
\usepackage[all]{xy}
\usepackage{hyperref}
\usepackage{bookmark}
\usepackage{bbm}

\newtheorem{theorem}{Theorem}[section]
\newtheorem{corollary}[theorem]{Corollary}
\newtheorem{lemma}[theorem]{Lemma}
\newtheorem{prop}[theorem]{Proposition}

\newtheorem{pregunta}[theorem]{Question}
\newtheorem{conjetura}[theorem]{Conjecture}

\newtheorem{ejemplo}[theorem]{Example}
\newtheorem{comentario}[theorem]{Remark}

\theoremstyle{definition}
\newtheorem{definicion}[theorem]{Definition}

\numberwithin{equation}{section}

\newcommand*{\myproofname}{Proof of the Lemma:}

\newcommand*{\myproofteo}{Proof of the Theorem:}

\newcommand*{\myproofprop}{Proof of the Proposition:}

\newcommand{\Q}{\mathbb{Q}}

\newcommand{\Z}{\mathbb{Z}}

\newcommand{\Ps}{\mathbb{P}}

\newcommand{\NS}{\mathrm{NS}}
\newcommand{\Pic}{\mathrm{Pic}}

\newcommand*{\shom}{\mathscr{H}\kern -.5pt om}
\newcommand*{\shext}{\mathscr{E}\kern -.5pt xt}

\newcommand*{\shtor}{\mathscr{T}\kern -.5pt or}

\begin{document}

\title{Seshadri constants and hyperelliptic curves on abelian varieties}
\author{Nelson Alvarado}
\address{ Departamento de Matemáticas, Facultad de Ciencias, Universidad de Chile, Las Palmeras 3425,  Santiago\\Chile}
\email{nelson.alvarado@ug.uchile.cl}

\begin{abstract}
 Given a principally polarized abelian variety, we give a sufficient condition for the van Geemen-van der Geer locus $\Gamma_{00}$ to be positive dimensional. Along the way, we prove a sharp Castelnuovo-type inequality for hyperelliptic curves on abelian varieties, and characterize the cases in which we have equality. These results are consequences of a general relation between Seshadri constants and the surjectivity of high order Gau\ss-Wahl maps on curves inside abelian varieties, which we also prove. Applying such relation to Abel-Jacobi curves we obtain new results about the surjectivity of higher Gau\ss-Wahl maps on smooth curves, sharpening a previous result of Bertram, Ein and Lazarsfeld. In the process, we discuss further questions and conjectures. 
\end{abstract}

\maketitle

\section{Introduction}

A classical problem in the theory of complex abelian varieties is to find characterizations of polarized Jacobians of smooth curves among all principally polarized abelian varieties (\cite{gru-survey}). Nowadays, even though there are many such characterizations (e.g \cite{Mats}, \cite{DPC} or \cite{GV-min-GV}), there are also many famous conjectures that remain unsolved. One of such conjectures is the $\Gamma_{00}$-conjecture of van Geemen and van der Geer (\cite{vgvdg}). Concretely, for an indecomposable principally polarized abelian variety $(A,\theta)$ and $\Theta$ a symmetric divisor representing $\theta,$ this conjecture predicts that the base locus $\Gamma_{00}(A,\theta)$ of the linear system 
$$|I_{0}^{4}(2\Theta)| = \{D\in|2\Theta| : \mathrm{mult}_{0}(D) \geq 4\},$$
where $0\in A$ is the origin of the group-law, has positive dimension if and only if $(A,\theta)$ is the polarized Jacobian of a smooth projective curve (see \cite[2.a)]{BD-theta}, for a more geometric interpretation in terms of \emph{fausse tris\'ecantes} of the Kummer variety). 

As shown in \cite{grushevski-gamma00} and \cite[Théorème 2]{BD-theta}, the $\Gamma_{00}$-conjecture holds under additional genericity hypothesis. That is, for a very general principally polarized abelian variety (p.p.a.v.) we have $\Gamma_{00}(A,\theta) = \{0\}.$ In contrast, in this article we show that the presence of low-degree hyperelliptic \footnote{For this we mean an irrational curve that admits a $2:1$ map to $\Ps^{1}.$ In particular, in this article elliptic curves are considered to be hyperelliptic} curves in $(A,\theta)$ ensures that $\Gamma_{00}(A,\theta)$ has positive dimension. Concretely, for $k,m\in\Z_{>0}$ we can consider the \emph{continuous} linear systems
$$\{I_{0}^{k}(m\Theta)\}:=\{D\equiv_{\mathrm{num}} m\Theta : \mathrm{mult}_{0}D \geq k\}.$$
For example, we obviously have that the base locus $\Gamma_{00}^{\mathrm{num}}(A,\theta)$ of the continuous system $\{I_{0}^{4}(2\Theta)\}$ is contained in $\Gamma_{00}(A,\theta).$ In this setting, our result is the following:
\begin{theorem}
\label{A}
Let $(A,\theta)$ be an indecomposable principally polarized abelian variety. Let $k,m$ be positive integers and assume that there exists a smooth hyperelliptic curve $C\subset A$ such that $2m(\theta\cdot C)< k(1+g(C))$. Then $$\dim \mathrm{Bs}(\{I_{0}^{k}(m\Theta)\}) > 0.$$
In particular, if there exists a smooth hyperelliptic curve with $(\theta\cdot C)\leq g(C),$ then 
$$\dim \Gamma_{00}(A,\theta)\geq \dim \Gamma_{00}^{\mathrm{num}}(A,\theta) > 0.$$
\end{theorem}
This result means, in particular, that either:
\begin{enumerate}
\item The $\Gamma_{00}$-conjecture is false in general, or
\item For any indecomposable p.p.a.v. $(A,\theta)$ and any smooth hyperelliptic curve $C$ contained in $A$ we have $(\theta\cdot C)\geq g(C).$ Moreover, equality would hold if and only if $(A,\theta)$ is the polarized Jacobian of an hyperelliptic curve $D$ (possibly $D\not\simeq C$). This would be an unexpected hyperelliptic refinement of the well known Ran's inequality (\cite[Theorem 3]{Ran}).
\end{enumerate}

Now, turning to a more detailed exposition, we first stress the fact that our proof of Theorem \ref{A} is not direct (indeed, it would be interesting to have a direct proof) and that along the way we obtain independent results about the geometry of hyperelliptic curves inside abelian varieties. Concretely, we use Debarre's observation \cite{deb}, which relates the locus $\Gamma_{00}^{\mathrm{num}}$ (or, more generally, the loci $\mathrm{Bs}(\{I_{0}^{k}(m\Theta)\}$) with the Seshadri constant $\varepsilon(A,\theta)$ of the polarized variety $(A,\theta)$ (we refer to Section 2 for its precise definition). More precisely, in \cite[Remark 2]{deb}, it is observed that $\Gamma_{00}^{\mathrm{num}}(A,\theta)$ is positive-dimensional as soon as $\varepsilon(A,\theta) < 2.$ In this context, our actual result (which already implies Theorem \ref{A}) is the following:
\begin{theorem}
\label{B}
Let $(A,\theta)$ be a polarized abelian variety and $C\subset A$ a smooth hyperelliptic curve contained in $A.$ We have that 
$$\varepsilon(\theta)\hspace{0.1cm}\leq\hspace{0.1cm} \dfrac{2(\theta\cdot C)}{g(C)+1},$$
where $\varepsilon(\theta)$ is the Seshadri constant of $\theta.$
\end{theorem}
 
On the other hand, a now well known result of Nakamaye (\cite[Theorem 1.1]{Nakamaye}, see also \cite{analogs}), says that for any polarization $\theta$ on an abelian variety $A$, we have that $\varepsilon(A,\theta)\geq 1$ and equality is attained only in the case that 
\begin{equation}
\label{tipo-especial}
(A,\ell)\simeq(E,\theta_{E})\boxtimes(B,m), 
\end{equation}
where $(E,\theta_{E})$ is a principally polarized elliptic curve. In particular, as a consequence of Theorem \ref{B} we obtain the following Castelnuovo-type inequality:
\begin{corollary}
\label{castelnuovo}
Let $(A,\theta)$ be a polarized abelian variety and let $C\subset A$ be a smooth hyperelliptic curve. Then we have that
$$g(C)\hspace{0.1cm}\leq\hspace{0.1cm} 2(\theta\cdot C) - 1.$$
Moreover, equality holds if and only if $(A,\theta)$ is as in \eqref{tipo-especial} via an isomorphism that identifies $C$ with $E\times\{b\},$ for some $b\in B$.
\end{corollary} 
To frame this result, it is worth to point out that for arbitrary curves inside polarized abelian varieties, there are some Castelnuovo-type inequalities, which are quadratic in the degree (\cite{degree-curves},\cite{PPo-castelnuovo}), but it is well known that they are not sharp. Our inequality has the advantage that it is sharp and linear on the degree, but has the obvious limitation that it holds just for smooth hyperelliptic curves. 

We conclude this introduction with a couple of words about the proofs of our results. The rough idea is that, as proved in \cite{jets}, the Seshadri constant of a polarized abelian variety $(A,\theta)$ is related to the asymptotic behaviour of the Gau\ss-Wahl maps $\gamma_{m\Theta,m\Theta}^{p}$ of powers of line bundles defining the polarization (we refer to Section 2 for their precise definition). On the other hand, we are able to show that, asymptotically, the surjectivity of Gau\ss-Wahl maps on $A$ can be transfered to the surjectivity of Gau\ss-Wahl maps on curves inside $A.$ Finally, thanks to Bertram-Ein-Lazarsfeld work (\cite{BEL}), there are many non-surjective Gau\ss-Wahl maps on smooth hyperelliptic curves. More precisely, if for $p\in\Z_{\geq 0},$ we write
$$e_{p}(C) : = \max\left\{e\in\mathbb{Z}_{\geq 0} : \gamma_{L,L}^{p} \hspace{0.1cm}\text{is not surjective for every $L\in\Pic(C)$ with $\mathrm{deg}L = e$}\right\},$$
then Bertram, Ein and Lazarsfeld show that 
\begin{equation}
\label{desigualdad-ep}
e_{p}(C)\leq (p+1)(g(C)+1)+g(C)-1
\end{equation}
with equality if and only if $C$ is hyperelliptic. With this in mind, Theorem \ref{B} is a direct consequence of the following statement, which we prove in Section 3:
\begin{theorem}
\label{general}
Let $(A,\theta)$ be a polarized abelian variety and let $C\subset A$ be a smooth curve contained in $A.$ Then we have that
$$\varepsilon(\theta)\leq 2(\theta\cdot C)\cdot\liminf_{p} \frac{p+1}{e_{p}(C)}.$$
\end{theorem}

Applying this result (or rather its proof) to the case of a smooth curve embedded inside its polarized Jacobian via an Abel-Jacobi map, we obtain a new result about Gau\ss-Wahl maps on curves. Concretely, in Section 4 we prove:
\begin{corollary}
\label{gausscurvas}
Let $C$ be a smooth curve and $L_{1},L_{2}\in\Pic(C)$. Assume that $g(C)\geq 5$ and that $C$ is not hyperelliptic nor bielliptic. Fix a positive integer $p\geq 2$ and assume that $\deg L_{1} =\deg L_{2} \geq g(C)(p+g(C)).$ Then the Gau\ss-Wahl map $\gamma_{L_1,L_2}^{p}$ is surjective.    
\end{corollary}
  
This article structures as follows: in Section 2 we recall some preliminaries about Gau\ss-Wahl maps on curves and the theory of jets-separation thresholds introduced in \cite{jets}, as well as its relation with Seshadri constants. In particular, we recall the relation between Theorem \ref{B} and Theorem \ref{A}. In Section 3 we prove Theorem \ref{general} and Theorem \ref{B} (the later being a direct consequence of the former). In particular, we explain how this statement is a higher-dimensional version of the fact that smooth hyperelliptic curves on abelian surfaces have genus at most five (e.g \cite[Theorem 2.8]{borowka-Ortega}) and we discuss the Castelnuovo-type inequality given by Corollary \ref{castelnuovo}. In Section 4 we specialize our methods to the case of Abel-Jacobi curves and prove Corollary \ref{gausscurvas}. In Section 5 we give some concluding remarks and mention a couple of elementary observations about the condition $(\theta\cdot C)\leq g(C)$. In particular, there we observe that, under that condition, $(A,\theta)$ is the polarized Jacobian of $C$ if and only if the generalized Brill-Noether locus
$$W_{C,A}^{1}(\Theta):=\{\alpha\in\Pic^{0}A : h^{0}(C,\mathcal{O}_{C}(\Theta_{\alpha}))\geq 2\}$$
has the \emph{expected codimension} as a degeneracy locus (see Proposition \ref{fibras-suma}).

\vspace{0.2cm}

\textbf{Acknowledgements:} 

\vspace{0.2cm}


\section{Preliminaries}

In this section we recall the preliminary materials that we will need. We start by fixing some notation.  Let $A$ be an abelian variety defined over an algebraically closed field $k$ of characteristic zero. By \emph{a polarization} on $A$ we mean an ample class $\ell\in\NS(A).$ For $A$ as before, we write 
$$\Pic^{0}A = \{L\in\Pic(A) : t_{x}^{\ast}L\simeq L \hspace{0.1cm}\text{for all $x\in A$}\},$$
where $t_{x}:A\rightarrow A$ is the morphism given by $z\mapsto x+z.$ We will write $\hat{A}$ for the dual of $A,$ that is, $\hat{A}$ is the abelian variety parametrizing traslation-invariant line bundles on $A.$ We fix $\mathcal{P}$ a normalized Poincar\'e bundle on $A\times\hat{A}$ and for $\alpha\in\hat{A},$ write $P_{\alpha}$ for the line bundle on $A$ parametrized by $\alpha$ (equivalently, $P_{\alpha}\in\Pic^{0}A$ is the fiber of $\mathcal{P}$ over $\alpha$). 

\subsection{Cohomological jump loci and vanishing thresholds}

For a coherent sheaf $\mathcal{F}$ on an abelian variety $A,$ we consider the cohomological jump loci
$$V^{i}(\mathcal{F}):= \left\{\alpha\in\hat{A} : H^{i}(A,\mathcal{F}\otimes P_{\alpha})\neq 0 \right\}$$
and we say that $\mathcal{F}$ is GV if 
$$\operatorname{codim}_{\hat{A}} V^{i}(\mathcal{F}) \geq i \hspace{0.5cm}\forall i\geq 0.$$
We also say that $\mathcal{F}$ is IT(0) if $V^{i}(\mathcal{F}) = \emptyset$ for $i>0.$
For this work, the following two examples are relevant:
\begin{ejemplo}[Example 2.5 in \cite{jets}]
\label{jets-ejemplo}
Let $0\in A$ be the origin of the group law and write $I_{0}$ for its ideal sheaf. Let $L$ be an ample line bundle on $A$ and $p\in\Z_{\geq 0}.$ Then $L$ separates $p$-jets at a general point $x$ (i.e the restriction map $H^{0}(A,L)\rightarrow H^{0}(A,L\otimes\mathcal{O}_{A}/I_{x}^{p+1})$ is surjective) if and only if $L\otimes I_{0}^{p+1}$ is GV.   
\end{ejemplo}

\begin{ejemplo}
\label{jac-ejemplo}
Let $C$ be a smooth curve embedded in its principally polarized Jacobian via an Abel-Jacobi embedding. Then $I_{C}(\Theta)$ is GV (\cite[Proposition 4.4]{regAVI}). Conversely, if $C$ is an integral curve generating an indecomposable principally polarized abelian variety and $I_{C}(\Theta)$ is GV, then $C$ has minimal cohomology class (\cite[Theorem B]{GV-min-GV}) and thus $(A,\theta)$ is the polarized Jacobian of $C.$ Moreover, by Theorem 4.1 in \cite{regAVI}, the sheaf $I_{C}(2\Theta)$ is IT(0).
\end{ejemplo}

Remarkably, in \cite[\S 5]{cohrank} the notion of being GV is extended to the setting of $\Q$-twisted sheaves. To explain this, recall that a $\Q$-twisted object $\mathcal{F}\left<t\ell\right>$ is the class of a pair $(\mathcal{F},t\ell),$ where $\mathcal{F}\in\operatorname{Coh}(A),$  $t$ is a rational number and $\ell\in\NS(A)$ is a polarization, under the equivalence relation generated by identifying $(\mathcal{F}\otimes L^{\otimes m},t\ell)$  with $(\mathcal{F},(t+m)\ell)$ for every integer $m$ and for every $L$ representing $\ell.$ In this context, we have: 
\begin{definicion}[\S 5 in \cite{cohrank}, Definition 3.1 in \cite{jets}]
\begin{enumerate}
\item[]
\item We say that the $\Q$-twisted sheaf $\mathcal{F}\left<t\ell\right>$ is GV if the sheaf $b_{A}^{\ast}\mathcal{F}\otimes L^{\otimes ab}$ is GV for a (any) representation of $t$ as a fraction $a/b,$ where $b_{A}:A\rightarrow A$ is the multiplication by $b$ isogeny.
\item Given a polarization $\ell,$ the vanishing threshold of $\mathcal{F}$ with respect to $\ell$ is the real number
$$\nu_{\ell}(\mathcal{F}) = \inf\left\{t\in\Q : \mathcal{F}\left<t\ell\right>\hspace{0.1cm}\text{is GV}\right\}.$$
\end{enumerate}
\end{definicion}

Inspired by \cite[\S 8]{cohrank} and Example \ref{jets-ejemplo} above, in \cite{jets} the $p$-jets separation threshold is defined as the number
$$\beta_{p}(\ell) := \nu_{\ell}(I_{0}^{p+1})= \inf\{t\in\Q : I_{0}^{p+1}\left<t\ell\right>\hspace{0.1cm}\text{is GV}\}.$$

\subsection{Seshadri constants and continuous linear systems}

Given a projective variety $A,$ a line bundle $L$ on $A$ and $x\in A,$ the \emph{Seshadri constant of $L$ at $x$} is the real number
\begin{equation}
\label{seshadridef}
\varepsilon(L,x):= \inf\left\{\frac{(L\cdot C)}{\mathrm{mult}_{x}(C)} : C\subset A\hspace{0.1cm}\text{irreducible curve with $x\in C$}\right\}.
\end{equation}
The Seshadri constant encodes the positivity of $L$ around $x.$ If $A$ is an abelian variety, then this invariant does not depend on the choice of the point $x$ and thus we obtain a measure of the \emph{global} positivity of $L.$ In this case, we simply write $\varepsilon(\ell)$ instead of $\varepsilon(L,x),$ where $\ell\in\NS(A)$ is the class of $L.$ Now, Theorem 4.1 in \cite{jets} asserts that
\begin{equation}
\label{seshadri-limit}
\varepsilon(\ell) = \lim_{p\to\infty} \frac{p+1}{\beta_{p}(\ell)} = \sup_{p} \frac{p+1}{\beta_{p}(\ell)}.
\end{equation}
On the other hand, by a theorem of Campana and Peternell (\cite[Theorem 2.3.18]{PositivityI}), there exists a positive-dimensional subvariety $V\subseteq A$ containing $x$ such that 
\begin{equation}
\label{campana}
\varepsilon(L,x) = \left(\frac{(L^{\dim V}\cdot V)}{\mathrm{mult}_{x} V}\right)^{1/\dim V}.
\end{equation}
If $V$ is a subvariety that achieves this equality, we say that $V$ \emph{computes} $\varepsilon(L,x).$ In \cite{deb}, Debarre observes that a variety computing the Seshadri constant of a polarized abelian variety must be contained in certain base loci of the continuous linear systems mentioned in the introduction. More precisely, recall that for $k,m\in\Z_{>0}$ we write
$$\{I_{0_A}^{k}(m\Theta)\} = \{D\equiv_{\mathrm{num}} m\Theta : \mathrm{mult}_{0_A}D\geq k\},$$
where $\Theta$ is a symmetric divisor representing $\theta$ and $0_{A}\in A$ is the origin of the group law. In this context, we have the following:
\begin{prop}[Lemma 1 in \cite{deb}]
\label{lema-deb}
Let $(A,\theta)$ be a polarized abelian variety and let $V\subseteq A$ be a subvariety that computes $\varepsilon(\theta) = \varepsilon(\mathcal{O}_{A}(\Theta),0_{A}).$ Assume that $\varepsilon(\theta)<k/m$ for $k,m\in\Z_{>0}.$ Then
$$V\subseteq \mathrm{Bs}(\{I_{0_A}^{k}(m\Theta)\}).$$
\end{prop}

In particular, we see that Theorem \ref{A} is a direct consequence of Theorem \ref{B}.

\subsection{Gau\ss-Wahl maps}

For line bundles $L,M$ on a smooth projective variety $X$ and $p\in\Z_{\geq 0},$ the $p$-th Gau\ss-Wahl map associated to the pair $(L,M)$ is the natural linear map 
$$\gamma_{L,M}^{p}: H^{0}(X\times X,(L\boxtimes M)\otimes I_{\Delta}^{p})\longrightarrow  H^{0}(X\times X,(L\boxtimes M)\otimes I_{\Delta}^{p}/I_{\Delta}^{p+1}),$$
where $\Delta\subset X\times X$ is the diagonal. Moreover, recall that, since $X$ is smooth, the target is isomorphic to $H^{0}(X,L\otimes M\otimes S^{p}\Omega_{X}).$ Note also that if $X$ is a curve, then $I_{\Delta}^{p}$ is the line bundle $\mathcal{O}_{X\times X}(-p\Delta).$ As customary, we write 
$$\operatorname{Rel}_{X}^{p}(L,M):= H^{0}(X\times X,(L\boxtimes M)\otimes I_{\Delta}^{p})$$
for the source of $\gamma_{L,M}^{p}.$

We have the following:
\begin{theorem}[Theorem 1.4 in \cite{jets}]
\label{surjGW}
Let $L,M$ be ample and algebraically equivalent line bundles on an abelian variety. Write $\ell$ for the class of $L$ and $M$ in $\NS(A).$ Let $c,d$ be positive integers and assume that 
$$\beta_{p}(\ell) < \frac{cd}{c+d}.$$
Then the $p$-th Gau\ss-Wahl map $\gamma_{cL,dM}^{p}$ is surjective. 
\end{theorem}

On the other hand, in \cite{BEL}, Bertram, Ein and Lazarsfeld proved the following:

\begin{theorem}
Let $C$ be a smooth curve of genus $h.$  Let $L,M$ be line bundles on $C.$ 
\begin{enumerate}[label=\alph*)]
\item Assume that $\operatorname{deg} L,\operatorname{deg} M \geq (p+1)(h+1)$ with $\operatorname{deg} L + \operatorname{deg} M > 2(p+1)(h+1).$ Then the Gau\ss-Wahl map $\gamma_{L,M}^{p}$ is surjective.  
\item Assume that $C$ is hyperelliptic. Write $\iota:C\rightarrow C$ for the hyperelliptic involution and $R$ for the ramification divisor. If $\operatorname{deg} L,\operatorname{deg} M \geq (p+1)(h+1)$ and  
$$H^{0}(C,\omega_{C}((p+1)R)\otimes L^{\vee}\otimes\iota^{\ast}M^{\vee})\neq 0,$$
then the Gau\ss-Wahl map $\gamma_{L,M}^{p}$ is NOT surjective
\end{enumerate}
\end{theorem}

Regarding b), it is worth to point out the following more manageble consequence:
\begin{corollary}
\label{notsurjhyp}
Let $L$ be a line bundle on a smooth hyperelliptic curve of genus $h.$ Assume that 
$$\operatorname{deg} L \leq(p+1)(h+1)+r,$$
with $0\leq r\leq h-1.$ Then the Gau\ss-Wahl map $\gamma_{L,L}^{p}$ is not surjective. 
\end{corollary}

In this context, the invariant
$$e_{p}(C) = \max\{e\in\Z_{\geq 0} : \gamma_{L,L}^{p} \hspace{0.1cm}\text{is not surjective for any $L\in\Pic(C)$ with $\operatorname{deg}L = e$}\}$$
considered in the introduction, satisfy the inequality 
\begin{equation}
\label{BEL-inequality}
e_{p}(C) \leq (p+1)(g(C)+1) + g(C) -1,
\end{equation}
with equality if and only if $C$ is hyperelliptic. 

\section{Gau\ss-Wahl maps on curves in abelian varieties}

Let $A$ be an abelian variety and $Y\subseteq A$ a smooth subvariety. As in \cite[4.]{Faro-Spelta}, we can consider the obvious commutative diagram induced by restriction:
\begin{equation*}
\xymatrix{ \operatorname{Rel}_{A}^{p}(L,M) \ar[rrr]^-{\gamma_{L,M}^{p}}\ar[dd] & & & H^{0}\left(L\otimes M\otimes\mathrm{Sym}^{p}\Omega_{A}\right) \ar[d]^-{u} \\
& & &  H^{0}(\left.L\right|_{Y}\otimes\left.M\right|_{Y}\otimes\mathrm{Sym}^{p}\left.\Omega_{A}\right|_{Y})\ar[d]^-{v} \\
\operatorname{Rel}_{Y}^{p}(\left.L\right|_{Y},\left.M\right|_{Y})\ar[rrr]_-{\gamma^{p}_{\left.L\right|_{Y},\left.M\right|_{Y}}} & & & H^{0}\left(\left.L\right|_{Y}\otimes\left.M\right|_{Y}\otimes\mathrm{Sym}^{p}\Omega_{Y}\right)}
\end{equation*}
In particular, if $\gamma_{L,M}^{p}, u$ and $v$ are surjective, then $\gamma^{p}_{\left.L\right|_{Y},\left.M\right|_{Y}}$ is also surjective. The point here is that the surjectivity of $u$ and $v$ can be controlled by certain specific vanishings. Concretely, as $\Omega_{A}$ is free, it follows that the surjectivity of $u$ is ensured whenever 
\begin{equation}
\label{anulamiento-ideal}
H^{1}\left(A,I_{Y}\otimes L\otimes M\right) = 0,
\end{equation}
where $I_{Y}$ is the ideal sheaf of $Y.$ The surjectivity of $v$ is a bit more subtle. Write
$$Q_{0}(p) = \ker\left[\operatorname{Sym}^{p}\left.\Omega_{A}\right|_{Y}\twoheadrightarrow\operatorname{Sym}^{p}\Omega_{Y}\right].$$
In this context, the surjectivity of $v$ is ensured by the vanishing
\begin{equation}
\label{anulamiento-nucleo}
H^{1}\left(Y, Q_{0}(p)\otimes \left.(L\otimes M)\right|_{Y}\right) = 0.
\end{equation}
Of course, the sheaf $Q_{0}(p)$ is in principle a bit misterious. However, we can be a bit more precise. Indeed, taking symmetric product to the conormal sequence, we get the following long exact sequence (\cite[I.4.3.1.7]{Illusie}):
$$0\longrightarrow\omega_{Y}\otimes S^{p-c}\left.\Omega_{A}\right|_{Y}\longrightarrow\bigwedge^{c-1}N_{Y/A}^{\vee}\otimes S^{p-c+1}\left.\Omega_{A}\right|_{Y}\longrightarrow \cdots$$
$$\longrightarrow\bigwedge^{2}N_{Y/A}^{\vee}\otimes S^{p-2}\left.\Omega_{A}\right|_{C}\longrightarrow N_{Y/A}^{\vee}\otimes S^{p-1}\left.\Omega_{A}\right|_{Y}\longrightarrow S^{p}\left.\Omega_{A}\right|_{Y}\longrightarrow S^{p}\Omega_{Y}\longrightarrow 0,$$
where $c = \operatorname{codim}_{A}(Y).$ If we write
$$Q_{1}(p) = \ker\left[N_{Y/A}^{\vee}\otimes S^{p-1}\left.\Omega_{A}\right|_{Y}\longrightarrow S^{p}\left.\Omega_{A}\right|_{Y}\right],$$
then we have an exact sequence 
$$0\longrightarrow Q_{1}(p)\longrightarrow N_{Y/A}^{\vee}\otimes S^{p-1}\left.\Omega_{A}\right|_{Y}\longrightarrow Q_{0}(p)\longrightarrow 0.$$
Tensoring with $\left.(L\otimes M)\right|_{Y}$ and taking cohomology we obtain the exact sequence of cohomology groups on $Y$:
$$H^{1}\left(N_{Y/A}^{\vee}\otimes S^{p-1}\left.\Omega_{A}\right|_{Y}\otimes\left.(L\otimes M)\right|_{Y}\right)\rightarrow H^{1}\left(Q_{0}(p)\otimes \left.(L\otimes M)\right|_{Y}\right)\rightarrow H^{2}\left(Q_{1}(p)\otimes \left.(L\otimes M)\right|_{Y}\right).$$
In particular, if $Y = C$ is a curve, as $\Omega_{A}$ is trivial, in order to obtain \eqref{anulamiento-nucleo}, it is enough to have the following vanishing, which does not depend on $p$:
\begin{equation}
\label{anulamiento-conormal} 
H^{1}\left(C,N_{C/A}^{\vee}\otimes \left.(L\otimes M)\right|_{C}\right) = 0.
\end{equation}

Summarizing, we have:

\begin{prop}
\label{sobreyectividad-subvariedades}
Let $C$ be a smooth curve contained in an abelian variety $A.$ Let $L,M$ be ample line bundles on $A$ such that 
$$H^{1}(A,I_{C}\otimes L\otimes M) = H^{1}(C,N_{C/A}^{\vee}\otimes\left.(L\otimes M)\right|_{C}) = 0.$$ 
Then for any $p\in\Z_{\geq 0},$ the surjectivity of $\gamma_{L,M}^{p}$ implies the surjectivity of $\gamma^{p}_{\left.L\right|_{Y},\left.M\right|_{Y}}.$
\end{prop}

At this point, we prove Theorem \ref{A} and \ref{B} from the introduction:

\begin{theorem}
\label{principal}
Let $(A,\theta)$ be a (non-necessarily principally) polarized abelian variety and $C\subset A$ be a smooth curve contained in $A.$ Let $d:=(\theta\cdot C)$ be the degree of $C$ with respect to $\theta.$ Then we have that
$$\varepsilon(\theta)\leq 2d\cdot\liminf_{p} \frac{p+1}{e_{p}(C)}.$$
Moreover, if $C$ is hyperelliptic then we get 
$$\varepsilon(\theta) \leq \frac{2d}{h+1},$$
where $h$ denotes the genus of $C.$
\end{theorem}

\begin{proof}

Write $L$ for a line bundle on $A$ representing the polarization $\theta.$ By Serre vanishing, there exists a positive integer $m_{0}$ such that 
$$H^{1}(A,I_{C}\otimes L^{\otimes 2m}) = H^{1}(C,N_{C/A}^{\vee}\otimes\left.L^{\otimes 2m}\right|_{C}) = 0$$
for any $m\geq m_{0}.$ Let $d = (L\cdot C)$ be the degree $L$ in $C.$ For $p\in\Z_{\geq 0},$ define:
$$m(p):= \max\{m\in\Z_{\geq 0} : md \leq e_{p}(C)\}$$
(such maximum exists, by \eqref{BEL-inequality}). Moreover, we claim that for all $p\gg 0$ we have that $m(p)\geq m_{0}.$ Indeed, the opposite would be that for all line bundles $T$ of degree $m_{0}d$ on $C,$ the Gau\ss-Wahl map $\gamma_{T,T}^{p}$ is surjective for infinitely many $p.$ However, the source of $\gamma_{T,T}^{p}$ is included in $H^{0}(C\times C,T\boxtimes T),$ while the dimension of the target increases arbitrarily with $p,$ which proves the claim. In particular, by Proposition \ref{sobreyectividad-subvariedades}, the Gau\ss-Wahl map (on $A$) $\gamma_{m(p)L,m(p)L}^{p}$ is not surjective and thus, by Theorem \ref{surjGW} we have the inequality 
$$\beta_{p}(L) \geq \frac{m(p)}{2} \geq \frac{1}{2}\left(\frac{e_{p}(C)}{d}-1\right).$$
Since this inequality holds for arbitrarily big $p,$ the result follows from \eqref{seshadri-limit} by dividing by $p+1$ and taking $\limsup.$

If $C$ is hyperelliptic, the claimed inequality directly follows from Proposition \ref{notsurjhyp} above, which asserts that $e_{p}(C) = (p+1)(h+1)+h-1$ in that case.  

\end{proof}

\begin{comentario}
Since for a very general p.p.a.v. we have $\Gamma_{00}(A,\theta) = \{0\}$  (\cite{BD-theta}), the above theorem also proves that a very general p.p.a.v. does not contain low-degree hyperelliptic curves. This gives a weak version of Pirola's theorem (\cite{Pirola}).
\end{comentario}

Note that if $C$ is a smooth curve contained in $A,$ then, by definition \eqref{seshadridef}, we have the loose inequality $\varepsilon(\ell)\leq (\ell\cdot C).$ On the other hand, by \eqref{campana}, in order to find good upper bounds for $\varepsilon(\ell),$ one needs to look for higher-dimensional subvarieties that are quite singular at a given point. The above theorem says that, in the case of abelian varieties, the intrinsic geometry of smooth curves inside $A$ gives non-trivial upper bounds for the Seshadri constant. Roughly speaking, for more special curves, the numbers $e_{p}(C)$ should be closer to Bertram-Ein-Lazarsfeld bound, and thus this gives constrains for the possible embeddings of curves in abelian varieties. In this context, it is worth to raise the following question about smooth curves:
\begin{pregunta}
\label{pregunta-gauss}
Let $C$ be a smooth projective curve. Find upper bounds for $e_{p}(C)$ in terms of intrinsic invariants of $C$ (e.g the gonality or Clifford index). In other words: construct non-surjective $p$-th order Gau\ss-Wahl maps on $C.$
\end{pregunta}

This question is addressed for $p=1$ in \cite[3.(D)]{pareschi-gauss}. Nevertheless, it is not obvious if the results exposed there are sharp and, more importantly, it is not clear whether it is possible to extend them to higher-order Gau\ss-Wahl maps. On the other hand, applying the theorem to $(JC,\theta_{C})$ and an Abel-Jacobi curve, we see that, at least asymptotically, the behavior of the Gau\ss-Wahl maps on a curve is controlled by the Seshadri constant of its polarized Jacobian. In the next section we refine the proof of Theorem \ref{principal} to this case, and render this principle effective and thus we obtain new results about the surjectivity of these maps. 

We now give a couple of illustrating examples:

\begin{ejemplo}
Note that if $(JC,\theta_{C})$ is the polarized Jacobian of a smooth hyperelliptic curve $C$ and we embed $C$ via an Abel-Jacobi map, then Theorem \ref{B} above says that 
\begin{equation*}
\varepsilon(JC,\theta_{C})\hspace{0.1cm}\leq\hspace{0.1cm} \frac{2h}{h+1}\hspace{0.1cm}<\hspace{0.1cm} 2,
\end{equation*}
which is a well known result (\cite[Proposition (ii)]{L}. Moreover, in \cite[Theorem 7]{deb} it is shown that in that case we have equality). 
\end{ejemplo}

\begin{ejemplo}
In \cite{deb}, Debarre conjectured that an indecomposable p.p.a.v.$(A,\theta)$ with $\dim A\geq 4$ and $\varepsilon(\theta)<2$ is necessarily the polarized Jacobian of a smooth hyperelliptic curve. In \cite[Corollary 1.3]{lozovanu}, Lozovanu showed that this conjecture is true if $\dim A = 4.$ In particular, Theorem \ref{principal} above implies that if $C$ is a smooth hyperelliptic curve contained in an indecomposable 4-dimensional p.p.a.v., then $g(C)\leq (\theta\cdot C),$ with equality if and only if $(A,\theta)$ is the polarized jacobian of a smooth hyperelliptic curve $D.$ Moreover, if there is equality then $\varepsilon(JD,\theta_{D})\leq 2g(C)/(g(C)+1),$ with equality if and only if $D=C.$ 
\end{ejemplo}

\begin{ejemplo}[$\subset$ Theorem 2.8 in \cite{borowka-Ortega}]
\label{caso-sup}
Let $A$ be an abelian surface with $\NS(A)=\Z\cdot\ell,$ where $\ell$ is ample. Let $C$ be any smooth hyperelliptic curve contained in $A.$ Then $g(C)\leq 5.$
\end{ejemplo}

\begin{proof}

First, as $\ell$ generates $\NS(A),$ $\ell$ is primitive and thus it has type $(1,e)$ for some $e\in\Z_{\geq 1}.$ Now, since $\NS(A)$ has rank 1, it follows that $A$ is simple and thus $g(C)\geq 2.$ In particular, $C$ is an ample divisor and thus $C\equiv m\ell$ for some $m\in\Z_{>0}.$ In particular, $\varepsilon(\mathcal{O}_{A}(C)) = m\varepsilon(\ell).$ On the other hand, by \cite[Theorem 6.1]{bauer-supalg} we know that 
$$\varepsilon(\ell)\geq \frac{2e}{\sqrt{2e+1}}.$$
In particular, since $(C\cdot\ell) = m(\ell^2) = 2me,$ combining with Theorem \ref{principal} above, we obtain the inequality
$$(g(C)+1)^{2}\leq 4(2m^{2}e+m^{2}) = 4(2g(C)-2+m^2)$$
and thus $g(C)^{2}-6g(C)+(9-4m^{2})\leq 0.$ It follows that 
$$3-2m \leq g(C)\leq 3+2m.$$
Now, as $g(C) = 1+m^{2}e,$ we have $1+m^{2}\leq g(C)\leq 3+2m,$ thus $(m-1)^{2}\leq 3$ and $m\in\{1,2\}.$ For $m = 2,$ we obtain $1+4e\leq 7,$ that is $e\leq 3/2$ and thus $e=1.$ In particular, $g(C) = 5$ in this case. On the other hand, for $m=1$ we simply obtain $e\leq 4,$ which implies $g(C)\leq 5,$ as we wanted to show.  

\end{proof}

To conclude this section, we prove and discuss the consequences of the following:

\begin{corollary}
\label{castelnuovo-ineq}
Let $C$ be a smooth hyperelliptic curve contained in a polarized abelian variety $(A,\theta).$ Then $g(C)\leq 2(\theta\cdot C) -1,$ with equality if and only if $(A,\theta)$ is isomorphic to $(E,\theta_{E})\boxtimes(B,m),$ where $\dim E=1$ and $\theta_{E}$ is principal. Moreover,  the isomorphism identifies $C$ with $E\times b,$ for some $b\in B.$ 
\end{corollary}

\begin{proof}
The result follows from Theorem \ref{principal} since $\varepsilon(\theta)\geq 1.$ Now, equality is attained if and only if $\varepsilon(\theta) = 1,$ which happens if and only if $(A,\theta)$ has at least one principally polarized curve factor. In this context, the only claim that still needs proof is the fact that $C$ must be of the form $E\times b.$ To prove this, write $A = \prod_{i=1}^{r}E_{i} \times B^{\prime},$ where $r\geq 1,$ $\dim E_{i} =1,$ $\left.\theta\right|_{E_i}$ principal and $B^{\prime}$ does not have principally-polarized curve factors. Let $C\subset A$ a smooth curve with $g(C) = 2(\theta\cdot C)-1.$ First, we note that the projection $\mathrm{pr}_{i}:E_{i}\rightarrow C$ must be dominant for some $i_0.$ Indeed, if this is not the case, then $C = e\times C^{\prime}$ for some curve $C^{\prime}$ with $(\left.\theta\right|_{B^\prime}\cdot C^{\prime}) = (\theta\cdot C).$ This forces $\varepsilon(\left.\theta\right|_{B^\prime}) = 1$ and hence that $B^{\prime}$ contains a principally polarized curve factor, which is absurd. Thus, we write $A = E\times B,$ where $E=E_{i_0}$ and $B = \prod_{i\neq i_0} E_{i} \times B^{\prime}.$ Note that the degree of the projection $\mathrm{pr} : C\to E$ is $d_{1}:=(\mathrm{pr}_{i_0}^{\ast}\theta_{E_i}\cdot C),$ which satisfies $d_{1}\leq d$ with equality if and only if $C$ is contracted by the projection $A\to B.$ Finally, since $C$ is hyperelliptic, we also have a $2:1$ map $C\to\Ps^{1}.$ Next, we note that since the targets are $\Ps^{1}$ and an elliptic curve, and one of the maps has degree 2, they can not both simultaneously factor through a non-trivial morphism. In particular, from Severi-Castelnuovo inequality (\cite[Theorem 3.5.(iii)]{Ac} we obtain:
$$h=g(C)\leq 2\cdot 0 + d_{1}\cdot 1 +(d_{1}-1)\cdot(2-1) = 2d_{1}-1\leq 2d-1 = h,$$
from which we get $d_{1}=d$ and thus that $C$ is contracted by $\mathrm{pr}:A\to B,$ that is, $C = E\times b$ for some $b\in B.$
   
\end{proof}

Finally, we specialize our result to the case in which $C$ is in a multiple of the minimal class:
\begin{corollary}
Let $(A,\theta)$ be a principally polarized $a$-dimensional abelian variety and $m\in\mathbb{Z}_{\geq 1}.$ Write $\xi_{min} = \theta^{a-1}/(a-1)!\in H^{2(a-1)}(A,\Z).$ If $C$ is a smooth hyperelliptic curve such that $[C] = m\xi_{min}\in H^{2(a-1)}(A,\mathbb{Z}),$ then: 
\begin{enumerate}[label = \alph*)]
\item $g(C) \leq 2ma - 1,$ and
\item $\varepsilon(\theta)\hspace{0.1cm}<\hspace{0.1cm} 2m.$ 
\end{enumerate}
\end{corollary}  

\begin{proof}

Part a) follows immediately from Corollary \ref{castelnuovo-ineq} since a curve with $[C] = m\xi_{min}$ satisfies $(\theta\cdot C) = ma.$ For the second part, recall that a curve in $m\xi_{min}$ generates $A$ (\cite[Corollary II.2 and II.3]{Ran}) and thus $g(C)\geq a.$ In particular, 
$$\varepsilon(\theta)\leq \frac{2ma}{g(C)+1} = 2m\cdot\frac{a}{g(C)+1}<2m.$$ 

\end{proof}

It is worth to point out that, motivated by Welter's result (\cite{Welters-twice}) about curves in $2\xi_{min},$ in \cite{JSZ} Jow, Sauvaget and Zelaci conjectured that the genus of curves with class $m\xi_{min}$ should satisfy $g(C)\leq ma + P(m),$ for a polynomial $P.$ In this context, since the $a$-coefficient in our upper bound is linear in $m$, part a) of the above corollary can be thought as the first higher-dimensional result about this Conjecture (with the obvious limitation that it is valid just for hyperelliptic curves). 

Note also that part b) of the above corollary for $m=1$ could be thought as a weak and hyperelliptic of the Matsusaka-Ran criterion. Indeed: an indecomposable p.p.a.v. containing a smooth curve in $\xi_{min}$ is the polarized Jacobian of such curve. On the other hand, from Proposition \ref{lema-deb} it follows that if the $\Gamma_{00}$-conjecture is true, then the only i.p.p.a.v's with $\varepsilon(\theta)<2$ are the Jacobians of smooth hyperelliptic curves. 


The author expects that the results exposed in this section can be generalized to curves of arbitrary gonality. Concretely, inspired by the Jacobian case (\cite[Proposition]{L}), we propose the following
\begin{conjetura}
Let $(A,\theta)$ be a polarized abelian variety. Let $C$ be a smooth genus $h$ curve contained in $A.$ Write $c$ for the gonality of $C$ and $d = (\theta\cdot C)$ for its $\theta$-degree. Then 
$$\varepsilon(\theta)\leq \frac{cd}{c+h-1}.$$
\end{conjetura}

\section{Gau\ss-Wahl maps on curves}

In this section we specialize to the case of a smooth curve embedded in its polarized Jacobian via an Abel-Jacobi map. In this way we obtain new information about the Gau\ss-Wahl maps on smooth curves using the Seshadri constant of its Jacobian. Concretely, we prove:

\begin{theorem}
 Let $C$ be a smooth curve and $L_{1},L_{2}\in\Pic(C)$. Assume that $g(C)\geq 5$ and that $C$ is not hyperelliptic nor bielliptic. Fix a positive integer $p\geq 2$ and assume that $\deg L_{1} =\deg L_{2} \geq g(C)(p+g(C)).$ Then the Gau\ss-Wahl map $\gamma_{L_1,L_2}^{p}$ is surjective.     
\end{theorem}

 \begin{proof}
First, recall that, by an standard multiplier-ideals construction (\cite[Proposition 6.3]{jets}), we have
$$\beta_{p}(\theta_{C}) \leq \frac{p+g(C)}{\varepsilon(\theta_C)}.$$
Now, if $C$ falls under the hypothesis of the Theorem, then $\varepsilon(\theta_{C})>2,$ by \cite[Theorem 7]{deb}, and thus
$$\beta_{p}(\theta_{C}) < \frac{p+g(C)}{2}.$$ 
In particular, if we write $L:=\mathcal{O}_{JC}((p+g(C))\Theta_{C}),$ then by Theorem \ref{surjGW} the Gau\ss-Wahl map $\gamma_{L\otimes\eta_{1},L\otimes\eta_{2}}^{p}$ (of line bundles on $JC$) is surjective for all $\eta_{1},\eta_{2}\in\Pic^{0}(JC).$

Next, note that by Lemma \ref{inductivogauss} below, we might assume that $\deg L_{i} = g(C)(p+g(C)).$ If we embed $C$ in $JC$ by means of an Abel-Jacobi map $\iota$, then $(\theta\cdot C) = g(C)$ and the restriction morphism $\iota^{\ast}:\Pic^{0}(JC)\rightarrow \Pic^{0}(C)$ is an isomorphism. It follows that $L_{i} = \iota^{\ast}(\mathcal{O}_{JC}((p+g(C))\Theta_{C})\otimes\eta_{i})$ for some $\eta_{i}\in\Pic^{0}(JC).$ By Proposition \ref{sobreyectividad-subvariedades}, to transfer the surjectivity, we need $$H^{1}(A,I_{C}(2(p+g)\Theta)\otimes\eta_{1}\otimes\eta_{2}) = 0.$$ Since $I_{C}(2\Theta_{C})$ is IT(0) (Example \ref{jac-ejemplo}), this vanishing is automatic as soon as $2(p+g)\geq 2,$ which holds for all $p\geq 0$ and $g\geq 1.$ 
Now, from the proof of Proposition \ref{sobreyectividad-subvariedades}, we see that the second thing that we need to consider in order to transfer the surjectivity, is the $H^{0}$ of the natural map $\left.S^{p}\Omega_{JC}\right|_{C}\otimes L_{1}\otimes L_{2}\rightarrow\omega_{C}^{\otimes p}\otimes L_{1}\otimes L_{2}.$ Now, in the Jacobian case, we have that $\Omega_{JC}$ is simply $H^{0}(C,\omega_{C})\otimes\mathcal{O}_{JC}$ and thus we are limited to consider the multiplication map $$S^{p}H^{0}(\omega_{C})\otimes H^{0}(L_{1}\otimes L_{2})\rightarrow H^{0}(\omega_{C}^{\otimes p}\otimes L_{1}\otimes L_{2}).$$ By the classical Mumford-Castelnuovo theorem (\cite[Theorem 6]{mumford2}), this map is surjective as soon as $\deg \omega_{C}^{\otimes p} \geq 2g$ and $2\deg L \geq 2g+1,$ i.e $2p(g-1)\geq 2g$ and $2g(p+g)>2g+1.$ The second condition is automatic for $p\geq 0$ and $g\geq 2$ (or $p\geq 1$ and $g\geq 1$) and the first one holds for $p\geq 2$ and $g\geq 2.$ The result then follows.

\end{proof}

\begin{lemma}
\label{inductivogauss}
Let $C$ be a smooth projective curve. Assume that for some $e\in\Z_{\geq 1}$ and for all $L_{1},L_{2}\in\Pic^{e}(C)$ the Gau\ss-Wahl map $\gamma_{L_1,L_2}^{p}$ is surjective. Then $\gamma_{M_1,M_2}^{p}$ is surjective for every $M_{1},M_{2}\in\Pic^{e+1}(C).$
\end{lemma}

\begin{proof}

Let $x,y\in C$ be two points. We have the following commutative diagram with exact columns:
\begin{equation*}
\xymatrix{ \mathrm{Rel}^{p}(M_{1}(-x),M_{2}(-y)) \ar[r]\ar[d]_{\gamma} & \mathrm{Rel}^{p}(M_1,M_2) \ar[d]^{\gamma} \\ 
 H^{0}(\omega_{C}^{\otimes p}\otimes M_{1}\otimes M_{2}(-x-y)) \ar[r]_{\cdot xy}\ar[d]_{0} & H^{0}(\omega_{C}^{\otimes p}\otimes M_{1}\otimes M_{2}) \ar[d]^{u} \\ 
 H^{1}((M_{1}(-x)\boxtimes M_{2}(-y))(-(p+1)\Delta)) \ar[r] & H^{1}((M_{1}\boxtimes M_{2})(-(p+1)\Delta)) }
\end{equation*}
where the $\gamma$'s are the corresponding Gau\ss-Wahl maps and the 0 at the left bottom comes from the surjectivity of $\gamma_{M_1(-x),M_2(-y)}^{p},$ which we have by hypothesis. By commutativity, this implies that 
$$\mathrm{im}(\cdot xy)\subset\ker u = \mathrm{im} \gamma_{M_{1},M_{2}}^{p}.$$ 
Now, $\deg (\omega_{C}^{\otimes p}\otimes M_{1}\otimes M_{2}) = 2(e+1 +p(g-1)) > 2.$ In particular, for any section $s\in H^{0}(\omega_{C}^{\otimes p}\otimes M_1 \otimes M_2)$ there exist $x,y\in C$ with $s(x) = s(y) = 0$ (here, possibly $x=y,$ i.e $s$ vanishes at $x$ with multiplicity two) . The previous argument shows then that $s$ lies in the image of $\gamma_{M_1,M_2}^{p}$ and thus, since $s$ is arbitrary, we conclude that $\gamma_{M_1,M_2}^{p}$ is surjective, as we wanted to show.

\end{proof}

Here we note that this result improves \cite{BEL} as soon as
$$g(g+p)\leq (p+1)(g+1)+g-1,$$
i.e as soon as $g(g -2) \leq p.$ 

\section{Concluding remarks}



In this section we study indecomposable principally polarized abelian varieties (i.p.p.a.v) $(A,\theta)$ containing smooth curves whose $\theta$-degree is bounded by its genus. First, we remark that if we do not ask $C$ to be hyperelliptic, then it is easy to construct examples of indecomposable p.p.a.v. $(A,\theta),$ not polarized Jacobians, containing smooth curves satisfying such inequality. Concretely: 

\begin{ejemplo}
Let $(A,\theta)$ be an $a$-dimensional p.p.a.v. Let $\Theta$ be a divisor representing $\theta.$ Then for $D_{1},...,D_{a-1}\in|2\Theta|$ general, the intersection $D_1\cap\cdots\cap D_{a-1}$ is a smooth curve $C$ (\cite[Teorema 2.8]{ottaviani}). This curve satisfies that  $$\omega_{C}\simeq\left.\mathcal{O}_{A}(2(a-1)\Theta)\right|_{C}.$$ In particular: $2g(C)-2 = 2(a-1)(\theta\cdot C)$ and thus 
$$(\theta\cdot C) = \frac{1}{a-1}(g(C)-1)<g(C).$$
Note, however, that for $a\geq 3$ the curve obtained in this way is never hyperelliptic. Indeed, for $a\geq 3$ we have that $2(a-1)\Theta$ is very ample and thus $\omega_{C}$ is also very ample.
\end{ejemplo}

\begin{ejemplo}
Consider a $2:1$ \'etale cover $C\to C^{\prime}$ of smooth curves. As customary, write $g(C^{\prime}) = a+1,$ thus $g(C) = 2a+1.$ Write $(P,\theta)$ for the polarized Prym-variety associated to this cover, thus $\dim P = a$. If $C$ is not hyperelliptic then the Abel-Prym map $C\to P$ is an embedding (\cite[Corollary 12.5.6]{CAV}) and by Welter's criterion (\cite[12.2.2]{CAV}) we have $[C] = 2\xi_{min}.$ In particular, $(\theta\cdot C) = 2a < 2a+1 = g(C).$     
\end{ejemplo}

In the rest of this section we provide some remarks and general implications of the existence of a smooth curve with $(\theta\cdot C)\leq g(C).$ To do this, we use the fact that, by Example \ref{jac-ejemplo}, $(A,\theta)$ is the polarized Jacobian of $C$ if and only if the sheaf $I_{C}(\Theta)$ is GV, where $\Theta$ is a (any) divisor representing the polarization $\theta.$ 

First, note that, twisting the exact sequence $0\to I_{C}(\Theta)\to\mathcal{O}_{A}(\Theta)\to\mathcal{O}_{C}(\Theta)\to 0$ by $\alpha\in\hat{A},$ we get 
\begin{equation}
\label{Vuno}
V^{1}(I_{C}(\Theta)) = \left\{\alpha\in\hat{A} :  H^{0}(\mathcal{O}_{A}(\Theta_{\alpha}))\rightarrow H^{0}(\mathcal{O}_{C}(\Theta_{\alpha}))\hspace{0.1cm}\text{is not surjective}\right\}
\end{equation}
and
\begin{equation}
\label{Vdos}
V^{2}(I_{C}(\Theta)) = V^{1}(\mathcal{O}_{C}(\Theta)).
\end{equation}
On the other hand, by Riemann-Roch on $C,$ we have that 
\begin{equation}
\label{RR}
h^{0}(C,\mathcal{O}_{C}(\Theta_{\alpha})) = 1 + (\theta\cdot C) - g(C) + h^{1}(C,\mathcal{O}_{C}(\Theta_{\alpha})).
\end{equation}
Using this, we prove:
\begin{lemma}
\label{LemaVi}
Let $C$ be a smooth curve contained in a principally polarized abelian variety $(A,\theta).$ 
\begin{enumerate}[label = \alph*)]
\item If $(\theta\cdot C) = g(C),$ then 
$$V^{2}(I_{C}(\Theta)) = W_{C,A}^{1}(\mathcal{O}_{A}(\Theta)) :=\left\{\alpha\in\hat{A} : h^{0}(C,\mathcal{O}_{C}(\Theta_{\alpha}))\geq 2\right\}.$$
\item If $(\theta\cdot C)\leq g(C)$ and $V^{2}(I_{C}(\Theta))\neq\hat{A},$ then $V^{1}(I_{C}(\Theta))\neq\hat{A}$ and $(\theta\cdot C) = g(C).$
\item $V^{1}(I_{C}(\Theta)) = \hat{A}$ if and only if $h^{0}(C,\mathcal{O}_{C}(\Theta_{\alpha}))\geq 2$ for all $\alpha\in\hat{A}.$
\end{enumerate}
\end{lemma} 

\begin{proof}

Part a) directly follows from \eqref{Vdos} and \eqref{RR}. Now, under the hypothesis on b), there exists $\alpha\in\hat{A}$ such that $h^{1}(\mathcal{O}_{C}(\Theta_{\alpha})) = 0$ while, on the other hand, in any case we have that $\Theta_{\alpha}$ is effective and thus $h^{0}(\mathcal{O}_{C}(\Theta_{\alpha}))\geq 1.$ In particular, it follows from \eqref{RR} that $(\theta\cdot C)= g(C)$ and $h^{0}(\mathcal{O}_{C}(\Theta_{\alpha})) = 1$ for general $\alpha\in\hat{A}.$ Now, for such general $\alpha,$ we have that the map $H^{0}(\mathcal{O}_{A}(\Theta_{\alpha}))\rightarrow H^{0}(\mathcal{O}_{C}(\Theta_{\alpha}))$ fails to be surjective if and only if it is zero, which happens if and only if $C\subset\Theta_{\alpha}.$ Since for general $\alpha$ we have that $C$ intersects $\Theta_{\alpha}$ transversely (for instance, by Kleiman's transversality Theorem, \cite[4. Corollary]{kleiman}), we conclude that $V^{1}(I_{C}(\Theta))\neq\hat{A}$ by \eqref{Vuno}. Part c) follows similarly.

\end{proof}

Now, recall that the generalized Brill-Noether locus
$$W_{C,A}^{1}(\mathcal{O}_{C}(\theta_{A})):=\{\alpha\in\hat{A} : h^{0}(C,\mathcal{O}_{C}(\Theta_{A})\otimes\iota^{\ast}P_{\alpha})\geq 2\}$$
mentioned in the introduction admits a description as the degeneracy locus of certain morphism of vector bundles. This is shown, for example, in \cite[5.1]{pareschi-generation}, but we recall the construction here just for the sake of self-containedness. Concretely, for a normalized Poincar\'e bundle $\mathcal{P}\in\Pic(A\times\hat{A})$ we consider the line bundle $\mathcal{P}_{C,A} = (\iota\times\mathrm{id})^{\ast}\mathcal{P}\in\Pic(C\times\hat{A})$ and define the functor $\Phi_{\mathcal{P}_{C,A}}:D^{b}(C)\rightarrow D^{b}(A)$ by 
$$\Phi_{\mathcal{P}_{C,A}}(\mathcal{F}) = \mathrm{Rpr}_{A\ast}\left(\mathrm{pr}_{C}^{\ast}(\mathcal{F})\otimes\mathcal{P}_{C,A}\right),$$
where $\mathrm{pr}_{C},\mathrm{pr}_{A}$ are the product projections from $C\times A.$ Now, let $E$ be a divisor in $C$ with $\operatorname{deg} E \gg 0.$ Then 
$$H^{1}(C,\mathcal{O}_{C}(\Theta_{A}+E)\otimes\iota^{\ast}P_{\alpha}) = 0 \hspace{0.4cm}\forall\alpha\in\hat{A},$$
and thus $\Phi_{\mathcal{P}_{C,A}}(\mathcal{O}_{C}(\Theta_{A}+E))$ is a locally free sheaf (concentrated in degree zero) by Grauert's theorem. Similarly, since $\dim E = 0,$ the same is true for $\Phi_{\mathcal{P}_{C,A}}(\left.\mathcal{O}_{C}(\Theta_{A}+E)\right|_{E}).$ In this context, one considers the restriction morphism $\mathrm{res}_{E}:\mathcal{O}_{C}(\Theta_{A}+E)\rightarrow\left.\mathcal{O}_{C}(\Theta_{A}+E)\right|_{E}$ and apply the functor $\Phi_{\mathcal{P}_{C,A}},$ to get a morphism of vector bundles. Note that the source of this morphism has rank $\chi(\mathcal{O}_{C}(\Theta_{A}+E))$ while the target has rank $\operatorname{deg} E.$ Moreover, by the usual base change in cohomology we have that
$$H^{0}(C,\mathcal{O}_{C}(\Theta_{A})\otimes\iota^{\ast}P_{\alpha}) = \ker(\Phi_{\mathcal{P}_{C,A}}(\mathrm{res}_{E})\otimes k(\alpha)).$$
In particular, 
$$W_{C,A}^{1}(\mathcal{O}_{C}(\Theta_{A})) = \left\{\alpha\in\hat{A} : \operatorname{rk}\left(\Phi_{\mathcal{P}_{C,A}}(\mathrm{res}_{E})\otimes k(\alpha)\right) \leq\chi(\mathcal{O}_{C}(\Theta_{A}+E)) - 2 \right\},$$
the $(\chi(\mathcal{O}_{C}(\Theta_{A}+E)) - 2)$-degeneracy locus of $\Phi_{\mathcal{P}_{C,A}}(\mathrm{res}_{E}).$ The expected codimension (in $\hat{A}$) of this locus is $2(1-(\theta_{A}\cdot C)+g(C)).$ In particular, if $(\theta_{A}\cdot C) = g(C)$, we have that 
$$\mathrm{exp.codim}_{\hat{A}} W_{C,A}^{1}(\mathcal{O}_{C}(\Theta_{A})) = 2.$$

In this context, we can prove the following: 

\begin{prop}
\label{fibras-suma}
Let $(A,\theta_{A})$ be an indecomposable principally polarized abelian variety. Assume that there exists a smooth curve $C$ generating $A$ and satisfying $(\theta\cdot C) = g(C).$ Write $\iota:C\hookrightarrow A$ for the inclusion and $m:\iota^{\ast}\Pic^{0}A\times W_{g(C)}^{1}(C)\rightarrow W_{g(C)}(C)$ for the sum map. The following are equivalent:
\begin{enumerate}[label = \roman*)]
\item $(A,\theta_{A})$ is the polarized Jacobian of $C$
\item $W_{C,A}^{1}(\mathcal{O}_{A}(\Theta_{A}))$ has the expected codimension (namely, two) for a (any) divisor $\Theta_{A}$ representing $\theta_{A}$.
\item The intersection $\iota^{\ast}\Pic^{0}(A)\cap(W_{g(C)}^{1}(C)\otimes\mathcal{O}_{C}(-\Theta_{A}))$ (taken inside $JC$) has the expected dimension (namely $\dim A -2$)
\item The closed set 
$$Z_{\dim A -1}(m):=\{L\in W_{g(C)}(C) : \dim m^{-1}(L)\geq \dim A -1\}$$
is empty. 
\end{enumerate}
\end{prop}

\begin{proof}

If $(A,\theta)$ is not the polarized Jacobian of $C,$ then $I_{C}(\Theta)$ is not GV. Under the hypothesis $(\theta\cdot C)= g(C)$ we have, by the above lemma, that either:
\begin{enumerate}
\item $V^{1}(I_{C}(\Theta)) = V^{2}(I_{C}(\Theta)) = \hat{A},$ or
\item $\dim V^{2}(I_{C}(\Theta)) = \dim A -1$ and $V^{1}(I_{C}(\Theta))\neq\hat{A}$ 
\end{enumerate}
that is, the fact that $I_{C}(\Theta)$ is not GV is caused solely by $V^{2}(I_{C}(\Theta)).$ In any case, by Lemma \ref{LemaVi}a) we have 
$$\iota^{\ast}V^{2}(I_{C}(\Theta)) = \iota^{\ast}\Pic^{0}A \cap \left(W_{h}^{1}(C)\otimes\mathcal{O}_{C}(-\Theta)\right),$$
where $h=g(C)$ and the intersection is taken inside $\Pic^{0}C = JC.$ Now, for $\eta\in JC,$ the intersection
\begin{equation}
\label{int-eta}
\left(\eta\otimes\iota^{\ast}\Pic^{0}A\right)\cap\left(W_{h}^{1}(C)\otimes\mathcal{O}_{C}(-\Theta)\right)
\end{equation} 
is exactly the fiber over $\eta$ of the difference map 
$$\delta: \left(W_{h}^{1}(C)\otimes\mathcal{O}_{C}(-\Theta)\right)\times\iota^{\ast}\Pic^{0}A \longrightarrow JC$$
which, since $\iota^{\ast}\Pic^{0}A$ is an abelian variety, is also isomorphic to the fiber over $\mathcal{O}_{C}(\Theta)\otimes\eta$ of the sum map 
$$m:\iota^{\ast}\Pic^{0}A \times W_{h}^{1}(C)\longrightarrow W_{h}(C).$$
In this context, the fact that $\dim V^{2}(I_{C}(\Theta)) \geq \dim A -1,$ means exactly that $\mathcal{O}_{C}(\Theta)$ belongs to the set $Z_{\dim A -1}(m).$ Indeed: $m^{-1}(\mathcal{O}_{C}(\Theta))$ contains all the elements of the form $(\iota^{\ast}P_{\alpha}^{\vee},\mathcal{O}_{C}(\Theta_{\alpha}))$ for $\alpha\in V^{2}(I_{C}(\Theta))$ and the map $\iota^{\ast}:\Pic^{0}A\to\Pic^{0}C$ is finite onto its image. The equivalence with (ii) follows from the fact that the dimension of \eqref{int-eta} is the same as the one of $W_{C,A}^{1}(\mathcal{O}_{A}(\Theta_{A})),$ simply because the former is the counter-image of the later via the map $\iota^{\ast},$ which is finite onto its image)

\end{proof}  

We now discuss the statement and proof of the above result: 
\begin{enumerate}[align=left, leftmargin=0pt, labelindent=\parindent,
listparindent=\parindent, labelwidth=0pt, itemindent=!]
\item Note that in Case 1 of the above proof we have that there is a translate of $\iota^{\ast}\Pic^{0}A$ contained in $W_{g(C)}^{1}(C).$ By \cite[Proposition 3.3]{debarrewd} this can only occur if $g(C)$ is at least $2\dim A +2.$  Moreover, by \cite[Corollary 3.9]{CLP}, equality occurs if and only there exists a $2:1$ morphism $f:C\to C^{\prime}$ to a smooth curve $C^{\prime}$ of genus $a=\dim A$ with $g(C)>2a+1$ and $\iota^{\ast}\Pic^{0}A = f^{\ast}JC^{\prime}.$ Now, if $C$ is assumed hyperelliptic, then by Severi-Castelnuovo inequality, we have 
$$2a+1<g(C) \leq 2\cdot 0 + 2\cdot a + (2-1)\cdot(2-1) = 2a+1,$$
which is absurd. Thus in this situation we have $g(C)\geq 2\dim A + 3.$  
\item The set $Z_{\dim A -1}(m)\subset W_{g(C)}(C)$ is itself invariant under translations by elements in $\iota^{\ast}\Pic^{0}A.$ Indeed: if $L\in Z_{\dim A-1}(m)$ and $\alpha\in\hat{A}$ then $m^{-1}(L\otimes\iota^{\ast}P_{\alpha})$ contains all the elements $(\iota^{\ast}P_{\beta+\alpha},M)$ for $(\iota^{\ast}P_{\beta},M)\in m^{-1}(L).$ A priori we do not know whether this set is contained in some $W_{g(C)}^{r}(C)$ for $r>0.$
\item By Kleiman's transversality Theorem (\cite[4. Corollary]{kleiman}), we know that if $\eta\in\Pic^{0}C$ is a  general element, then the intersection  \eqref{int-eta} has the expected dimension. In this setting, the condition $\dim V^{2}(I_{C}(\Theta))\geq \dim A -1$ means exactly that the origin $\mathcal{O}_{C}$ belongs to the closed set of $\Pic^{0}C$ for which Kleiman's theorem fails.  
\end{enumerate}

We conclude this article with a couple of comments about the case $d:=(\theta\cdot C)< g(C).$ 

In this situation, by Proposition 5.3(i) in \cite{degree-curves} (which directly follows from \cite[Proposition 3.3]{debarrewd}) we necessarily have that $d\geq 2\dim A.$ We can give a slightly more detailed description of what happens in this case. Concretely, by Lemma \ref{LemaVi}b) above, the hypothesis $d<g(C)$ ensures that $V^{2}(I_{C}(\Theta)) = \hat{A}.$ If $V^{1}(I_{C}(\Theta))$ is also equal to $\hat{A},$ then by Lemma \ref{LemaVi}c) we have that $W_{d}^{1}(C)$ contains a translate of $\iota^{\ast}\Pic^{0}A$ and thus $g(C)>d\geq 2\dim A +2,$ once again by \cite[Proposition 3.3]{debarrewd}. However, unlike the case $(\theta\cdot C) = g(C),$ it could happen that $V^{1}(I_{C}(\Theta))\neq\hat{A}.$ In this situation, for general $\alpha$ we have $h^{0}(\mathcal{O}_{C}(\Theta_{\alpha})) = 1$ and $g(C) = d + h^{1}(\mathcal{O}_{C}(\Theta_{\alpha})).$ In any case, a translate of $\iota^{\ast}\Pic^{0}A$ is contained in $W_{d}(C)$ and thus  $d\geq 2\dim A$. By Corollary 3.9 in \cite{CLP}, equality occurs if and only if $\iota^{\ast}\Pic^{0}A = f^{\ast}JC^{\prime}$ for a degree two morphism $f:C\to C^{\prime}$ to a curve $C^{\prime}$ of genus $a<(g(C)-1)/2.$ By Severi-Castelnuovo inequality, this can not happen if $C$ is assumed to be hyperelliptic. Thus:
\begin{prop}
Let $C$ be a smooth hyperelliptic curve contained in a principally polarized abelian variety $(A,\theta).$ Suppose that $(\theta\cdot C) < g(C).$ Then $(\theta\cdot C)\geq 2\dim A +1.$
\end{prop}

\bibliographystyle{alpha}
\bibliography{referencias-hyp}

@book {PositivityI,
    AUTHOR = {Lazarsfeld, Robert},
     TITLE = {Positivity in algebraic geometry. {I}},
    SERIES = {Ergebnisse der Mathematik und ihrer Grenzgebiete. 3. Folge. A
              Series of Modern Surveys in Mathematics [Results in
              Mathematics and Related Areas. 3rd Series. A Series of Modern
              Surveys in Mathematics]},
    VOLUME = {48},
      NOTE = {Classical setting: line bundles and linear series},
 PUBLISHER = {Springer-Verlag, Berlin},
      YEAR = {2004},
     PAGES = {xviii+387},
      ISBN = {3-540-22533-1},
   MRCLASS = {14-02 (14C20)},
  MRNUMBER = {2095471},
MRREVIEWER = {Mihnea Popa},
       DOI = {10.1007/978-3-642-18808-4},
       URL = {https://doi.org/10.1007/978-3-642-18808-4},
}

@incollection {BEL,
    AUTHOR = {Bertram, Aaron and Ein, Lawrence and Lazarsfeld, Robert},
     TITLE = {Surjectivity of {G}aussian maps for line bundles of large
              degree on curves},
 BOOKTITLE = {Algebraic geometry ({C}hicago, {IL}, 1989)},
    SERIES = {Lecture Notes in Math.},
    VOLUME = {1479},
     PAGES = {15--25},
 PUBLISHER = {Springer, Berlin},
      YEAR = {1991},
      ISBN = {3-540-54456-9},
   MRCLASS = {14H60},
  MRNUMBER = {1181203},
MRREVIEWER = {Enrique\ Arrondo},
       DOI = {10.1007/BFb0086260},
       URL = {https://doi.org/10.1007/BFb0086260},
}

@article{jets,
    author = {Alvarado, Nelson},
    title = {Jets Separation Thresholds, Seshadri Constants, and Higher Gauss–Wahl Maps on Abelian Varieties},
    journal = {International Mathematics Research Notices},
    volume = {2025},
    number = {9},
    pages = {rnaf107},
    year = {2025},
    month = {05},
    issn = {1073-7928},
    doi = {10.1093/imrn/rnaf107},
    url = {https://doi.org/10.1093/imrn/rnaf107},
    eprint = {https://academic.oup.com/imrn/article-pdf/2025/9/rnaf107/63018243/rnaf107.pdf},
}

@book {CAV,
    AUTHOR = {Birkenhake, Christina and Lange, Herbert},
     TITLE = {Complex abelian varieties},
    SERIES = {Grundlehren der mathematischen Wissenschaften [Fundamental
              Principles of Mathematical Sciences]},
    VOLUME = {302},
   EDITION = {Second},
 PUBLISHER = {Springer-Verlag, Berlin},
      YEAR = {2004},
     PAGES = {xii+635},
      ISBN = {3-540-20488-1},
   MRCLASS = {14-02 (14H37 14Kxx 32G20)},
  MRNUMBER = {2062673},
MRREVIEWER = {Fumio Hazama},
       DOI = {10.1007/978-3-662-06307-1},
       URL = {https://doi.org/10.1007/978-3-662-06307-1},
}

@incollection {deb,
    AUTHOR = {Debarre, Olivier},
     TITLE = {Seshadri constants of abelian varieties},
 BOOKTITLE = {The {F}ano {C}onference},
     PAGES = {379--394},
 PUBLISHER = {Univ. Torino, Turin},
      YEAR = {2004},
   MRCLASS = {14K05 (14C20)},
  MRNUMBER = {2112583},
MRREVIEWER = {Federica Galluzzi},
}

@article {L,
    AUTHOR = {Lazarsfeld, Robert},
     TITLE = {Lengths of periods and {S}eshadri constants of abelian
              varieties},
   JOURNAL = {Math. Res. Lett.},
  FJOURNAL = {Mathematical Research Letters},
    VOLUME = {3},
      YEAR = {1996},
    NUMBER = {4},
     PAGES = {439--447},
      ISSN = {1073-2780},
   MRCLASS = {14K05 (14H40 14H42)},
  MRNUMBER = {1406008},
MRREVIEWER = {H. Lange},
       DOI = {10.4310/MRL.1996.v3.n4.a1},
       URL = {https://doi.org/10.4310/MRL.1996.v3.n4.a1},
}

@article {bauer-supalg,
    AUTHOR = {Bauer, Thomas},
     TITLE = {Seshadri constants on algebraic surfaces},
   JOURNAL = {Math. Ann.},
  FJOURNAL = {Mathematische Annalen},
    VOLUME = {313},
      YEAR = {1999},
    NUMBER = {3},
     PAGES = {547--583},
      ISSN = {0025-5831,1432-1807},
   MRCLASS = {14C20 (14C21 14E25 14K05)},
  MRNUMBER = {1678549},
MRREVIEWER = {Tomasz\ Szemberg},
       DOI = {10.1007/s002080050272},
       URL = {https://doi.org/10.1007/s002080050272},
}

@article {borowka-Ortega,
    AUTHOR = {Bor\'{o}wka, Pawe\l and Ortega, Angela},
     TITLE = {Hyperelliptic curves on {$(1,4)$}-polarised abelian surfaces},
   JOURNAL = {Math. Z.},
  FJOURNAL = {Mathematische Zeitschrift},
    VOLUME = {292},
      YEAR = {2019},
    NUMBER = {1-2},
     PAGES = {193--209},
      ISSN = {0025-5874,1432-1823},
   MRCLASS = {14H40 (14H30)},
  MRNUMBER = {3968899},
MRREVIEWER = {Nathan\ Grieve},
       DOI = {10.1007/s00209-018-2174-2},
       URL = {https://doi.org/10.1007/s00209-018-2174-2},
}

@article {Pirola,
    AUTHOR = {Pirola, Gian Pietro},
     TITLE = {Curves on generic {K}ummer varieties},
   JOURNAL = {Duke Math. J.},
  FJOURNAL = {Duke Mathematical Journal},
    VOLUME = {59},
      YEAR = {1989},
    NUMBER = {3},
     PAGES = {701--708},
      ISSN = {0012-7094,1547-7398},
   MRCLASS = {14H40 (14H10 14K99)},
  MRNUMBER = {1046744},
MRREVIEWER = {Gang\ Xiao},
       DOI = {10.1215/S0012-7094-89-05931-0},
       URL = {https://doi.org/10.1215/S0012-7094-89-05931-0},
}

@article {Nakamaye,
    AUTHOR = {Nakamaye, Michael},
     TITLE = {Seshadri constants on abelian varieties},
   JOURNAL = {Amer. J. Math.},
  FJOURNAL = {American Journal of Mathematics},
    VOLUME = {118},
      YEAR = {1996},
    NUMBER = {3},
     PAGES = {621--635},
      ISSN = {0002-9327,1080-6377},
   MRCLASS = {14C20 (14K99)},
  MRNUMBER = {1393263},
MRREVIEWER = {Wolfgang\ Michael\ Ruppert},
       URL =
              {http://muse.jhu.edu/journals/american_journal_of_mathematics/v118/118.3nakamaye.pdf},
}

@article {degree-curves,
    AUTHOR = {Debarre, Olivier},
     TITLE = {Degrees of curves in abelian varieties},
   JOURNAL = {Bull. Soc. Math. France},
  FJOURNAL = {Bulletin de la Soci\'{e}t\'{e} Math\'{e}matique de France},
    VOLUME = {122},
      YEAR = {1994},
    NUMBER = {3},
     PAGES = {343--361},
      ISSN = {0037-9484,2102-622X},
   MRCLASS = {14K05 (14H40)},
  MRNUMBER = {1294460},
MRREVIEWER = {Juan-Carlos\ Naranjo},
       URL = {http://www.numdam.org/item?id=BSMF_1994__122_3_343_0},
}

@article {Ran,
    AUTHOR = {Ran, Ziv},
     TITLE = {On subvarieties of abelian varieties},
   JOURNAL = {Invent. Math.},
  FJOURNAL = {Inventiones Mathematicae},
    VOLUME = {62},
      YEAR = {1981},
    NUMBER = {3},
     PAGES = {459--479},
      ISSN = {0020-9910,1432-1297},
   MRCLASS = {14K05 (14H40 14K10)},
  MRNUMBER = {604839},
MRREVIEWER = {T.\ Matsusaka},
       DOI = {10.1007/BF01394255},
       URL = {https://doi.org/10.1007/BF01394255},
}

@article {Mats,
    AUTHOR = {Matsusaka, Teruhisa},
     TITLE = {On a characterization of a {J}acobian variety},
   JOURNAL = {Mem. Coll. Sci. Univ. Kyoto Ser. A. Math.},
  FJOURNAL = {Memoirs of the College of Science. University of Kyoto. Series
              A. Mathematics},
    VOLUME = {32},
      YEAR = {1959},
     PAGES = {1--19},
      ISSN = {0368-8887},
   MRCLASS = {14.00},
  MRNUMBER = {108497},
MRREVIEWER = {Andr\'{e} Weil},
       DOI = {10.1215/kjm/1250776695},
       URL = {https://doi.org/10.1215/kjm/1250776695},
}

@article {vgvdg,
    AUTHOR = {van Geemen, Bert and van der Geer, Gerard},
     TITLE = {Kummer varieties and the moduli spaces of abelian varieties},
   JOURNAL = {Amer. J. Math.},
  FJOURNAL = {American Journal of Mathematics},
    VOLUME = {108},
      YEAR = {1986},
    NUMBER = {3},
     PAGES = {615--641},
      ISSN = {0002-9327},
   MRCLASS = {14K25 (11F46 14H40 14K10 32G20)},
  MRNUMBER = {844633},
MRREVIEWER = {Emma Previato},
       DOI = {10.2307/2374657},
       URL = {https://doi.org/10.2307/2374657},
}

@incollection {grushevski-gamma00,
    AUTHOR = {Grushevsky, Samuel},
     TITLE = {A special case of the {$\Gamma_{00}$} conjecture},
 BOOKTITLE = {Liaison, {S}chottky problem and invariant theory},
    SERIES = {Progr. Math.},
    VOLUME = {280},
     PAGES = {223--231},
 PUBLISHER = {Birkh\"{a}user Verlag, Basel},
      YEAR = {2010},
      ISBN = {978-3-0346-0200-6},
   MRCLASS = {14H42 (14H40)},
  MRNUMBER = {2664656},
       DOI = {10.1007/978-3-0346-0201-3\{_}12}

@incollection {BD-theta,
    AUTHOR = {Beauville, Arnaud and Debarre, Olivier},
     TITLE = {Sur les fonctions th\^{e}ta du second ordre},
 BOOKTITLE = {Arithmetic of complex manifolds ({E}rlangen, 1988)},
    SERIES = {Lecture Notes in Math.},
    VOLUME = {1399},
     PAGES = {27--39},
 PUBLISHER = {Springer, Berlin},
      YEAR = {1989},
   MRCLASS = {14H42 (14H40 14K25)},
  MRNUMBER = {1034254},
MRREVIEWER = {Gerald E. Welters},
       DOI = {10.1007/BFb0095966},
       URL = {https://doi.org/10.1007/BFb0095966},
}

@article {regAVI,
    AUTHOR = {Pareschi, Giuseppe and Popa, Mihnea},
     TITLE = {Regularity on abelian varieties. {I}},
   JOURNAL = {J. Amer. Math. Soc.},
  FJOURNAL = {Journal of the American Mathematical Society},
    VOLUME = {16},
      YEAR = {2003},
    NUMBER = {2},
     PAGES = {285--302},
      ISSN = {0894-0347,1088-6834},
   MRCLASS = {14K12 (14K05)},
  MRNUMBER = {1949161},
MRREVIEWER = {H.\ Lange},
       DOI = {10.1090/S0894-0347-02-00414-9},
       URL = {https://doi.org/10.1090/S0894-0347-02-00414-9},
}

@article {GV-min-GV,
    AUTHOR = {Pareschi, Giuseppe and Popa, Mihnea},
     TITLE = {Generic vanishing and minimal cohomology classes on abelian
              varieties},
   JOURNAL = {Math. Ann.},
  FJOURNAL = {Mathematische Annalen},
    VOLUME = {340},
      YEAR = {2008},
    NUMBER = {1},
     PAGES = {209--222},
      ISSN = {0025-5831},
   MRCLASS = {14K12 (14F17)},
  MRNUMBER = {2349774},
MRREVIEWER = {Jungkai Alfred Chen},
       DOI = {10.1007/s00208-007-0146-7},
       URL = {https://doi.org/10.1007/s00208-007-0146-7},
}

@article {cohrank,
    AUTHOR = {Jiang, Zhi and Pareschi, Giuseppe},
     TITLE = {Cohomological rank functions on abelian varieties},
   JOURNAL = {Ann. Sci. \'{E}c. Norm. Sup\'{e}r. (4)},
  FJOURNAL = {Annales Scientifiques de l'\'{E}cole Normale Sup\'{e}rieure. Quatri\`eme
              S\'{e}rie},
    VOLUME = {53},
      YEAR = {2020},
    NUMBER = {4},
     PAGES = {815--846},
      ISSN = {0012-9593},
   MRCLASS = {14F08 (14K05)},
  MRNUMBER = {4157109},
MRREVIEWER = {Jungkai Alfred Chen},
       DOI = {10.24033/asens.2435},
       URL = {https://doi.org/10.24033/asens.2435},
}

@article {JSZ,
    AUTHOR = {Jow, Shin-Yao and Sauvaget, Adrien and Zelaci, Hacen},
     TITLE = {On the principally polarized abelian varieties with
              {$M$}-minimal curves},
   JOURNAL = {Matematiche (Catania)},
  FJOURNAL = {Le Matematiche},
    VOLUME = {72},
      YEAR = {2017},
    NUMBER = {2},
     PAGES = {87--98},
      ISSN = {0373-3505,2037-5298},
   MRCLASS = {14K12},
  MRNUMBER = {3731504},
MRREVIEWER = {Pawe\l\ Bor\'owka},
       DOI = {10.4418/2017.72.2.7},
       URL = {https://doi.org/10.4418/2017.72.2.7},
}

@article {Welters-twice,
    AUTHOR = {Welters, G. E.},
     TITLE = {Curves of twice the minimal class on principally polarized
              abelian varieties},
   JOURNAL = {Nederl. Akad. Wetensch. Indag. Math.},
  FJOURNAL = {Koninklijke Nederlandse Akademie van Wetenschappen.
              Indagationes Mathematicae},
    VOLUME = {49},
      YEAR = {1987},
    NUMBER = {1},
     PAGES = {87--109},
      ISSN = {0019-3577},
   MRCLASS = {14K05 (14H40)},
  MRNUMBER = {883371},
MRREVIEWER = {Ryuji\ Sasaki},
}

@article {kleiman,
    AUTHOR = {Kleiman, Steven L.},
     TITLE = {The transversality of a general translate},
   JOURNAL = {Compositio Math.},
  FJOURNAL = {Compositio Mathematica},
    VOLUME = {28},
      YEAR = {1974},
     PAGES = {287--297},
      ISSN = {0010-437X,1570-5846},
   MRCLASS = {14M15 (14N10)},
  MRNUMBER = {360616},
MRREVIEWER = {Joel\ Roberts},
}

@article {CLP,
    AUTHOR = {Ciliberto, Ciro and Lopes, Margarida Mendes and Pardini, Rita},
     TITLE = {Abelian varieties in {B}rill-{N}oether loci},
   JOURNAL = {Adv. Math.},
  FJOURNAL = {Advances in Mathematics},
    VOLUME = {257},
      YEAR = {2014},
     PAGES = {349--364},
      ISSN = {0001-8708,1090-2082},
   MRCLASS = {14H40 (14H51)},
  MRNUMBER = {3187653},
MRREVIEWER = {Sanjay\ Kumar\ Singh},
       DOI = {10.1016/j.aim.2014.02.024},
       URL = {https://doi.org/10.1016/j.aim.2014.02.024},
}

@article {debarrewd,
    AUTHOR = {Debarre, Olivier and Fahlaoui, Rachid},
     TITLE = {Abelian varieties in {$W^r_d(C)$} and points of bounded degree
              on algebraic curves},
   JOURNAL = {Compositio Math.},
  FJOURNAL = {Compositio Mathematica},
    VOLUME = {88},
      YEAR = {1993},
    NUMBER = {3},
     PAGES = {235--249},
      ISSN = {0010-437X,1570-5846},
   MRCLASS = {14H30 (11G30 14H25 14H40)},
  MRNUMBER = {1241949},
MRREVIEWER = {John\ B.\ Little},
       URL = {http://www.numdam.org/item?id=CM_1993__88_3_235_0},
}

@book {Ac,
    AUTHOR = {Accola, Robert D. M.},
     TITLE = {Topics in the theory of {R}iemann surfaces},
    SERIES = {Lecture Notes in Mathematics},
    VOLUME = {1595},
 PUBLISHER = {Springer-Verlag, Berlin},
      YEAR = {1994},
     PAGES = {x+105},
      ISBN = {3-540-58721-7},
   MRCLASS = {30F10 (14H55)},
  MRNUMBER = {1329541},
MRREVIEWER = {C.\ Maclachlan},
       DOI = {10.1007/BFb0073575},
       URL = {https://doi.org/10.1007/BFb0073575},
}

@article {pareschi-generation,
    AUTHOR = {Pareschi, Giuseppe},
     TITLE = {Generation and ampleness of coherent sheaves on abelian
              varieties, with application to {B}rill-{N}oether theory},
   JOURNAL = {Pure Appl. Math. Q.},
  FJOURNAL = {Pure and Applied Mathematics Quarterly},
    VOLUME = {20},
      YEAR = {2024},
    NUMBER = {5},
     PAGES = {2379--2413},
      ISSN = {1558-8599,1558-8602},
   MRCLASS = {14F08 (14H51)},
  MRNUMBER = {4829853},
MRREVIEWER = {Hugo\ Torres\ L\'{o}pez},
       DOI = {10.4310/pamq.241105233715},
       URL = {https://doi.org/10.4310/pamq.241105233715},
}

@book {Illusie,
    AUTHOR = {Illusie, Luc},
     TITLE = {Complexe cotangent et d\'{e}formations. {I}},
    SERIES = {Lecture Notes in Mathematics, Vol. 239},
 PUBLISHER = {Springer-Verlag, Berlin-New York},
      YEAR = {1971},
     PAGES = {xv+355},
   MRCLASS = {14B10 (14D15 14F20 55D20)},
  MRNUMBER = {491680},
MRREVIEWER = {Larry\ Breen},
}

@article{Faro-Spelta,
author = {Faro, Dario and Spelta, Irene},
title = {Gauss-Prym maps on Enriques surfaces},
journal = {Mathematische Nachrichten},
volume = {296},
number = {9},
pages = {4454-4462},
doi = {https://doi.org/10.1002/mana.202200287},
url = {https://onlinelibrary.wiley.com/doi/abs/10.1002/mana.202200287},
eprint = {https://onlinelibrary.wiley.com/doi/pdf/10.1002/mana.202200287},
year = {2023}
}

@article {analogs,
    AUTHOR = {Alvarado, Nelson and Pareschi, Giuseppe},
     TITLE = {Abelian varieties analogs of two results about algebraic
              curves},
   JOURNAL = {Math. Ann.},
  FJOURNAL = {Mathematische Annalen},
    VOLUME = {395},
      YEAR = {2026},
    NUMBER = {3},
     PAGES = {Paper No. 74, 28},
      ISSN = {0025-5831,1432-1807},
   MRCLASS = {14K05 (14H52)},
  MRNUMBER = {5081483},
       DOI = {10.1007/s00208-026-03505-6},
       URL = {https://doi.org/10.1007/s00208-026-03505-6},
}

@incollection {gru-survey,
    AUTHOR = {Grushevsky, Samuel},
     TITLE = {The {S}chottky problem},
 BOOKTITLE = {Current developments in algebraic geometry},
    SERIES = {Math. Sci. Res. Inst. Publ.},
    VOLUME = {59},
     PAGES = {129--164},
 PUBLISHER = {Cambridge Univ. Press, Cambridge},
      YEAR = {2012},
      ISBN = {978-0-521-76825-2},
   MRCLASS = {14H42 (14K05)},
  MRNUMBER = {2931868},
MRREVIEWER = {H.\ Lange},
}

@article {DPC,
    AUTHOR = {Schreieder, Stefan},
     TITLE = {Theta divisors with curve summands and the {S}chottky problem},
   JOURNAL = {Math. Ann.},
  FJOURNAL = {Mathematische Annalen},
    VOLUME = {365},
      YEAR = {2016},
    NUMBER = {3-4},
     PAGES = {1017--1039},
      ISSN = {0025-5831,1432-1807},
   MRCLASS = {14H42 (14H40 14K12 14K25)},
  MRNUMBER = {3521080},
MRREVIEWER = {H.\ Lange},
       DOI = {10.1007/s00208-015-1287-8},
       URL = {https://doi.org/10.1007/s00208-015-1287-8},
}

@article {PPo-castelnuovo,
    AUTHOR = {Pareschi, Giuseppe and Popa, Mihnea},
     TITLE = {Castelnuovo theory and the geometric {S}chottky problem},
   JOURNAL = {J. Reine Angew. Math.},
  FJOURNAL = {Journal f\"ur die Reine und Angewandte Mathematik. [Crelle's
              Journal]},
    VOLUME = {615},
      YEAR = {2008},
     PAGES = {25--44},
      ISSN = {0075-4102,1435-5345},
   MRCLASS = {14K05 (14H40)},
  MRNUMBER = {2384330},
MRREVIEWER = {H.\ Lange},
       DOI = {10.1515/CRELLE.2008.008},
       URL = {https://doi.org/10.1515/CRELLE.2008.008},
}

@article {pareschi-gauss,
    AUTHOR = {Pareschi, Giuseppe},
     TITLE = {Gaussian maps and multiplication maps on certain projective
              varieties},
   JOURNAL = {Compositio Math.},
  FJOURNAL = {Compositio Mathematica},
    VOLUME = {98},
      YEAR = {1995},
    NUMBER = {3},
     PAGES = {219--268},
      ISSN = {0010-437X,1570-5846},
   MRCLASS = {14H10 (14H60 14K99)},
  MRNUMBER = {1351829},
MRREVIEWER = {Angelo\ Lopez},
       URL = {http://www.numdam.org/item?id=CM_1995__98_3_219_0},
}

@incollection {mumford2,
    AUTHOR = {Mumford, David},
     TITLE = {Varieties defined by quadratic equations},
 BOOKTITLE = {Questions on {A}lgebraic {V}arieties ({C}.{I}.{M}.{E}., {III}
              {C}iclo, {V}arenna, 1969)},
    SERIES = {Centro Internazionale Matematico Estivo (C.I.M.E.)},
     PAGES = {29--100},
 PUBLISHER = {Ed. Cremonese, Rome},
      YEAR = {1970},
   MRCLASS = {14.01},
  MRNUMBER = {282975},
MRREVIEWER = {Y.\ Nakai},
}

@book{ottaviani,
  title={Variet{\`a} proiettive di codimensione piccola},
  author={Ottaviani, G.},
  isbn={9788879991193},
  series={Ist. nazion. di alta matematica F. Severi},
  url={https://books.google.cl/books?id=9xlVAAAACAAJ},
  year={1995},
  publisher={Aracne}
}

@article {lozovanu,
    AUTHOR = {Lozovanu, Victor},
     TITLE = {Seshadri constants of indecomposable polarized abelian
              varieties},
   JOURNAL = {J. Math. Soc. Japan},
  FJOURNAL = {Journal of the Mathematical Society of Japan},
    VOLUME = {78},
      YEAR = {2026},
    NUMBER = {2},
     PAGES = {343--364},
      ISSN = {0025-5645,1881-1167},
   MRCLASS = {14C20 (14B10 14K25 14M25)},
  MRNUMBER = {5063445},
       DOI = {10.2969/jmsj/94319431},
       URL = {https://doi.org/10.2969/jmsj/94319431},
}

\end{document}